\documentclass[12pt]{amsart}
\usepackage{amssymb, latexsym, amsmath, amsthm}
\usepackage{stmaryrd}
\usepackage{hyperref}
\usepackage{fullpage}
\usepackage{mathrsfs}

\numberwithin{equation}{section}
\newtheorem{theorem}{\bf Theorem}[section]

\newtheorem{lemma}[theorem]{\bf Lemma}

\theoremstyle{definition}

\newtheorem{definition}[theorem]{\bf Definition}

\numberwithin{equation}{section}
\makeatletter
\let\c@theorem\c@equation
\makeatother

\newcommand{\F}{\mathcal{F}}
\newcommand{\C}{\mathcal{C}}
\newcommand{\K}{\mathcal{K}}
\newcommand{\A}{\mathcal{A}}

\newcommand{\M}{\mathcal{M}}

\newcommand{\E}{\mathcal{E}}

\renewcommand{\phi}{\varphi}
\newcommand{\gen}[1]{\langle #1 \rangle}

\newcommand{\Aut}{\operatorname{Aut}}
\newcommand{\Out}{\operatorname{Out}}
\newcommand{\Hom}{\operatorname{Hom}}
\newcommand{\End}{\operatorname{End}}

\renewcommand{\int}{\operatorname{int}}

\newcommand{\Inn}{\operatorname{Inn}}
\newcommand{\Inndiag}{\operatorname{Inndiag}}

\newcommand{\norm}{\trianglelefteqslant}

\begin{document}

\title{Fusion systems with some sporadic $J$-components}

\author{Justin Lynd}
\address{Institute of Mathematics, University of Aberdeen, Fraser Noble
Building, Aberdeen AB24 3UE} 
\email{justin.lynd@abdn.ac.uk}

\author{Julianne Rainbolt}
\address{Department of Mathematics and Statistics, Saint Louis University, 220
North Grand Blvd., Saint Louis, MO 63103}
\email{rainbolt@slu.edu}

\subjclass[2000]{} 
\date{\today}
\thanks{The research of the first author was partially supported by NSA Young
Investigator Grant H98230-14-1-0312 and was supported by an AMS-Simons grant
which allowed for travel related to this work.}

\begin{abstract}
Aschbacher's program for the classification of simple fusion systems of ``odd''
type at the prime $2$ has two main stages: the classification of $2$-fusion
systems of subintrinsic component type and the classification of $2$-fusion
systems of $J$-component type. We make a contribution to the latter stage by
classifying $2$-fusion systems with a $J$-component isomorphic to the
$2$-fusion systems of several sporadic groups under the assumption that the
centralizer of this component is cyclic. 
\end{abstract}

\maketitle

\section{Introduction}

The Dichotomy Theorem for saturated fusion systems \cite[II
14.3]{AschbacherKessarOliver2011} partitions the class of saturated $2$-fusion
systems into the fusion systems of characteristic $2$-type and the fusion
systems of component type. This is a much cleaner statement than the
corresponding statement for finite simple groups, and it has a much shorter
proof. In the last few years, M. Aschbacher has begun work on a program to give
a classification of a large subclass of the $2$-fusion systems of component
type.  A memoir setting down the outline and first steps of such a program is
forthcoming \cite{AschbacherFSCT}, but see \cite{AschbacherICCM} for a survey
of some of its contents. The immediate goal is to give a simpler proof of
roughly half of the classification of the finite simple groups by carrying out
most of the work in the category of saturated $2$-fusion systems. 

Let $\F$ be a saturated fusion system over a finite $2$-group $S$, of which the
standard example is the fusion system $\F_S(G)$, where $G$ is a finite group
and $S$ is a Sylow $2$-subgroup of $G$.  A \emph{component} is a subnormal,
quasisimple subsystem. The system is said to be of \emph{component type} if
some involution centralizer in $\F$ has a component. The $2$-fusion systems of
\emph{odd type} consist of those of \emph{subintrinsic component type} and
those of $J$-\emph{component type}. This is a proper subclass of the $2$-fusion
systems of component type. In focusing attention on this restricted class, one
is expected to avoid several difficulties in the treatment of \emph{standard
form} problems like the ones considered in this paper.  By carrying out the
work in fusion systems, it is expected that certain difficulties within
the classification of simple groups of component type can be avoided, including
the necessity of proving Thompson's $B$-conjecture.

We refer to \cite{AschbacherFSCT} for the definition of a fusion system of
subinstrinsic component type, as it is not needed in this paper.  The fusion
system $\F$ is said to be of $J$-\emph{component type} if it is not of
subintrinsic component type, and there is a (fully centralized) involution $x
\in S$ such that the $2$-rank of $C_S(x)$ is equal to the $2$-rank of $S$, and
$C_\F(x)$ has a component. We shall call such a component in an involution
centralizer a $J$-\emph{component}. 

In this paper, we classify saturated $2$-fusion systems having a $J$-component
isomorphic to the $2$-fusion system of $M_{23}$, $J_3$, $McL$, or $Ly$ under
the assumption that the centralizer of the component is a cyclic $2$-group. A
similar problem for the fusion system of $L_2(q)$, $q \equiv \pm 1 \pmod{8}$
was treated in \cite{Lynd2015} under stronger hypotheses.

\begin{theorem}\label{T:main}
Let $\F$ be a saturated fusion system over the finite $2$-group $S$.  Suppose
that $x \in S$ is a fully centralized involution such that $F^*(C_\F(x)) \cong
Q \times \K$, where $\K$ is the $2$-fusion system of $M_{23}$, $J_3$, $McL$, or
$Ly$, and where $Q$ is a cyclic $2$-group. Assume further that $m(C_S(x)) =
m(S)$. Then $\K$ is a component of $\F$. 
\end{theorem}

Here $F^*(C_\F(x))$ is the generalized Fitting subsystem of the centralizer
$C_\F(x)$ \cite{AschbacherGeneralized}, and $m(S) := m_2(S)$ is the $2$-rank of
$S$ -- that is, the largest rank of an elementary abelian $2$-subgroup of $S$.
We mention that any fusion system having an involution centralizer with a
component isomorphic to $McL$ or $Ly$ is necessarily of subintrinsic component
type by \cite[6.3.5]{AschbacherFSCT}. This means that, when restricted to those
components, Theorem~\ref{T:main} gives a result weaker than is needed to fit
into the subintrinsic type portion of Aschbacher's program.  However, we have
included $McL$ and $Ly$ here because our arguments apply equally well in each
of the four cases.

There is no almost simple group with an involution centralizer having any of
these simple groups as a component, but the wreath product $G = (K_1 \times
K_2)\gen{x}$ with $K_1^x = K_2$ always has $C_G(x) = \gen{x} \times K$ with $K$
a component that is diagonally embedded in $K_1 \times K_2$. The strategy for
the proof of Theorem~\ref{T:main} is to locate a suitable elementary abelian
subgroup $F$ in the Sylow $2$-subgroup of $\K$, and then to show that the
normalizer in $S$ of $E := \gen{x}F$ has at least twice the rank as that of
$F$. Thus, the aim is to force a resemblance with the wreath product, in which
$N_G(\gen{x}F)$ modulo core is an extension of $F_1 \times F_2$ (with $F_i$ the
projection of $F$ onto the $i$th factor) by $\gen{x} \times \Aut_K(F)$. 
Lemma~\ref{L:homocyclic} is important for getting control of the extension of
$E$ determined by $N_\F(E)$ in order to carry out this argument. 

\subsection*{Acknowledgements} We would like to thank the Department of
Mathematics and Statistics at Saint Louis University and the Departments of
Mathematics at Rutgers University and Ohio State University for their
hospitality and support during mutual visits of the authors. We would also like
to thank R. Solomon and R. Lyons for helpful discussions, and an anonymous
referee for their comments and suggestions.

\section{Background on fusion systems}\label{S:background}

We assume some familiarity with notions regarding saturated fusion systems as
can be found in \cite{AschbacherKessarOliver2011} or \cite{CravenTheory},
although some items are recalled below. Most of our notation is standard.

Whenever $G$ is a group, we write $G^\#$ for the set of nonidentity elements of
$G$. If we wish to indicate that $G$ is a split extension of a group $A \norm
G$ by a group $B$, then we will write $G = A \cdot B$.  For $g \in G$, denote
by $c_g$ the conjugation homomorphism $c_g\colon x \mapsto x^g$ and its
restrictions.  Morphisms in fusion systems are written on the right and in the
exponent.  That is, we write $x^\phi$ (or $P^\phi$) for the image of an element
$x$ (or subgroup $P$) of $S$ under a morphism $\phi$ in a fusion system, by
analogy with the more standard exponential notation for conjugation in a group.

\subsection{Terminology and basic properties}\label{SS:basic}
Throughout this section, fix a saturated fusion system $\F$ over the $p$-group
$S$.  We will sometimes refer to $S$ as the \emph{Sylow subgroup} of $\F$.  For
a subgroup $P \leq S$, we write $\Aut_\F(P)$ for $\Hom_{\F}(P,P)$, and
$\Out_\F(P)$ for $\Aut_\F(P)/\Inn(P)$. Whenever two subgroups or elements of
$S$ are isomorphic in $\F$, we say that they are $\F$-\emph{conjugate}.  Write
$P^\F$ for the set of $\F$-conjugates of $P$. If $\E$ is a subsystem of $\F$ on
the subgroup $T \leq S$ and $\alpha\colon T \to S$ is a morphism in $\F$, the
\emph{conjugate} of $\E$ by $\alpha$ is the subsystem $\E^\alpha$ over
$T^\alpha$ with morphisms $\phi^\alpha := \alpha^{-1}\phi\alpha$ for $\phi$ a
morphism in $\E$.

We first recall some of the terminology for subgroups and common subsystems in
a fusion system.  
\begin{definition}\label{D:basic}
Fix a saturated fusion system over the $p$-group $S$, and let $P \leq S$. Then
\begin{itemize}
\item $P$ is \emph{fully $\F$-centralized} if $|C_S(P)| \geq |C_S(Q)|$ for all $Q \in P^\F$,
\item $P$ is \emph{fully $\F$-normalized} if $|N_S(P)| \geq |N_S(Q)|$ for all $Q \in P^\F$,
\item $P$ is $\F$-\emph{centric} if $C_S(Q) \leq Q$ for all $Q \in P^\F$,
\item $P$ is $\F$-\emph{radical} if $O_p(\Out_\F(P)) = 1$, 
\item $P$ is \emph{weakly $\F$-closed} if $P^\F = \{P\}$,
\item the \emph{centralizer} of $P$ in $\F$ is the fusion system $C_\F(P)$ over
$C_S(P)$ with morphisms those $\phi \in \Hom_\F(Q, R)$ such that there is an
extension $\tilde{\phi} \in \Hom_{\F}(PQ, PR)$ that restricts to the identity
on $P$,
\item the \emph{normalizer} of $P$ in $\F$ is the fusion system $N_\F(P)$ over
$N_S(P)$ with morphisms those $\phi \in \Hom_\F(Q, R)$ such that there is an
extension $\tilde{\phi}\colon PQ \to PR$ in $\F$ such that $P^{\tilde{\phi}} = P$. 
\end{itemize}
\end{definition}

We write $\F^f$ and $\F^c$ for the collections of fully $\F$-normalized and
$\F$-centric, respectively, and we write $\F^{fc}$ for the intersection of
these two collections.

Sometimes we refer to an \emph{element} $x$ of $S$ as being fully
$\F$-centralized, when we actually mean that the group $\gen{x}$ is fully
$\F$-centralized, especially when $x$ is an involution.  For example, this was
done in the the statement of the theorem in the introduction.

Whenever $P \leq S$, we write $\mathfrak{A}(P)$ for the set of $\alpha \in
\Hom_\F(N_S(P), S)$ such that $P^\alpha$ is fully $\F$-normalized. 

\begin{lemma}\label{L:fn}
For each $P \leq S$, $\mathfrak{A}(P)$ is not empty. Moreover, for each $Q \in
P^\F \cap \F^f$, there is $\alpha \in \mathfrak{A}(P)$ with $P^\alpha = Q$.
\end{lemma}
\begin{proof}
This is \cite[A.2(b)]{BrotoLeviOliver2003}, applied with $K = \Aut(P)$. 
\end{proof}

By a result of Puig, the centralizer $C_\F(P)$ is saturated if $P$ is fully
$\F$-centralized, and the normalizer $N_\F(P)$ is saturated if $P$ is fully
$\F$-normalized. We write $O_p(\F)$ for the (unique) largest subgroup $P$ of
$S$ satisfying $N_\F(P) = \F$, and $Z(\F)$ for the (unique) largest subgroup
$P$ of $S$ satisfying $C_\F(P) = \F$. We note that if $\F = \F_S(G)$ for some
finite group $G$ with Sylow $p$-subgroup $S$, then $O_p(G) \leq S$ is normal in
$\F$ so that $O_p(G) \leq O_p(\F)$, but the converse does not hold in general.

\subsection{The Model Theorem}\label{SS:model} A subgroup $P \leq S$ is
$\F$-centric if and only if $C_S(P) \leq P$ and $P$ is fully $\F$-centralized
\cite[I.3.1]{AschbacherKessarOliver2011}.  If $P$ is $\F$-centric and fully
$\F$-normalized, then the normalizer fusion system $\M := N_\F(P)$ is
\emph{constrained} -- that is, $O_p(\M)$ is $\M$-centric. By the Model Theorem
\cite[Proposition~III.5.10]{AschbacherKessarOliver2011}, there is then a unique
finite group $M$ up to isomorphism having Sylow $p$-subgroup $N_S(P)$ and such
that $O_{p'}(M) = 1$, $O_p(M) = O_p(\M)$, and $\F_{N_S(P)}(M) \cong \M$. Then
$M$ is said to be a \emph{model} for $\M$ in this case.

\subsection{Tame fusion systems}\label{SS:tame}

The main hypothesis of Theorem~\ref{T:main}
is that the generalized Fitting subsystem of the involution centralizer $\C$ is
the fusion system of a finite group $Q \times K$, where $Q$ is a cyclic
$2$-group, and $K$ is simple. In this situation, $C_\F(x)$ is itself the fusion
system of a finite group $C$ with $F^*(C) = Q \times K$, where $K \cong M_{23}$,
$McL$, 
$J_3$, or $Ly$, since each of these simple groups \emph{tamely realizes} its
$2$-fusion system \cite{AOV2012, OliverTameSpor}.  Roughly, a finite group
tamely realizes its fusion system if every automorphism of its fusion system is
induced by an automorphism of the group. Moreover, a fusion system is said to
be \emph{tame} if there is some finite group that tamely realizes it. We refer
to \cite{AOV2012} for more details. 

The importance of tameness in the context of standard form problems was pointed
out in \cite[\S\S1.5]{Lynd2015}. The discussion there is centered around the
notion of strong tameness, which was needed for proofs of the results of
\cite{AOV2012}, but the contents of \cite{Oliver2013, GlaubermanLynd2016} imply
that a fusion system is tame if and only if it is strongly tame.  Recently,
Oliver has established the following useful corollary of the results in
\cite{AOV2012}, which we state for our setup here. 

\begin{theorem}[{\cite[Corollary~2.5]{OliverReductions}}]
\label{T:tame}
Let $\C$ be a saturated fusion system over a $2$-group. Assume that $F^*(\C) =
O_2(\C)\K$, where $\K$ is simple and tamely realized by a finite simple group
$K$. Then $\C$ is tamely realized by a finite group $C$ such that $F^*(C) =
O_2(C)K$. 
\end{theorem}

Note that, upon application of Theorem~\ref{T:tame} to the involution
centralizer $\C = C_\F(x)$ in Theorem~\ref{T:main}, we have $O_2(C) = Q =
O_2(\C)$. Indeed, $O_2(C) \leq Q$ since $O_2(C)$ is normal in $\C$, and one
sees that $Q = C_S(K)$ by combining Lemma~1.12(c) of \cite{Lynd2015} with
Lemma~\ref{L:sporautprelim}(c) below. However, $O_2(C)$ is normal and
self-centralizing in $C_C(K)$ by properties of the generalized Fitting
subgroup, so that $C_C(K)/O_2(C)$ is a group of outer automorphisms of the
cyclic $2$-group $O_2(C)$, and so is itself a $2$-group. It follows that
$C_C(K) = C_S(K)$ is a normal $2$-subgroup of $C$ (since $K \trianglelefteq
C$), and hence $Q = C_S(K) \leq O_2(C)$. 

Thus, the effect of Theorem~\ref{T:tame} for our purposes is that we may work
in the group $C$, where $Q$ is a normal subgroup.  In particular, in the setup
of the Theorem~\ref{T:main}, the quotient $C/Q$ is isomorphic to a subgroup of
$\Aut(K)$ containing $\Inn(K)$, where $K$ is one of the simple groups appearing
in Theorem~\ref{T:main}. 

\section{Structure of the components}\label{S:structure}

In this section, we recall some properties of the simple systems appearing in
Theorem~\ref{T:main} that are required for the remainder.

\begin{lemma}\label{L:H1}
Let $G$ be $A_7$ or $GL_2(4)$, and $V$ a faithful ${\bf F}_2[G]$-module of
dimension $4$. Then 
\begin{enumerate}
\item[(a)] $G$ acts transitively on the nonzero vectors of $V$,
\item[(b)] $C_{GL(V)}(G) \leq G$,
\item[(c)] $H^1(G,V) = 0$, and 
\item[(d)] if $G$ acts on a homocyclic $2$-group $Y$ with $\Omega_1(Y) = V$,
then $Y = V$. 
\end{enumerate}
\end{lemma}
\begin{proof}
In each case, $V$ is irreducible. There is a unique such module for $GL_2(4)
\cong C_3 \times A_5$, namely the natural ${\bf F}_4[G]$-module considered as a
module over ${\bf F}_2$, and thus (a) holds in this case.  The module for $A_7$
is unique up to taking duals; clearly points (a) and (b) are independent of the
choice between these two modules, and (c) is independent of such a choice by
\cite[\S 17]{AschbacherFGTSecond}. Note that $A_7$ acts transitively on $V^\#$, which
can be seen by noting that a Sylow $7$-subgroup acts with exactly one fixed
point on $V^\#$, and a Sylow $5$-subgroup acts with no fixed points.  Point (b)
holds for $G = A_7$ by absolute irreducibility.  Similarly for $G = GL_2(4)$,
one has that $C_{GL(V)}(G) = \End_{{\bf F}_2[G]}(V)^\times = {\bf F}_4^\times$,
and so (b) follows in this case as $Z(G) \cong C_3$.  Point (c) for $G =
GL_2(4)$ holds because, by coprime action, $Z(G)$ (and so $G$) has a fixed
point on any $5$-dimensional module containing $V$ as a submodule (see \cite[\S
17]{AschbacherFGTSecond}).  Point (c) for $G = A_7$ holds, for example, by applying
\cite{AlperinGorenstein1972}  with $\mathscr{L} = \{L_0, L_1, L_2\}$, where
$L_1 = C_G((1,2,3))$, $L_2 = C_G((4,5,6))$, and $L_0 = L_1 \cap L_2$. The $L_i$
indeed satisfy the hypotheses of that theorem: $H^0(L_i, V) = 0$ because each
Sylow $3$-subgroup of $GL_4(2) \geq G$ has no nontrivial fixed point on $V$,
and $H^1(L_i,V) = 0$ using a similar argument via coprime action as above. 

We now turn to (d), which follows from a special case of a result of G. Higman
\cite[Theorem~8.2]{Higman1968}. This says that if $SL_2(4)$ acts faithfully on
a homocyclic $2$-group $Y$ in which an element of order $3$ acts without fixed
points on $Y^\#$, then $Y$ is elementary abelian. In case $G = GL_2(4)$, $V$ is
the natural module for $G$ and certainly $SL_2(4) \leq G$ has (with respect to
an appropriate basis) a diagonal element of order $3$ acting without fixed
points. In the case $G = A_7$, we have $G \leq A_8 \cong GL_4(2)$, and the
action of $G$ on $V$ is the restriction of the natural (or dual) action of
$GL_4(2)$.  Restriction of either one to $GL_2(4) \cong C_3 \times SL_2(4)
\cong C_3 \times A_5$ shows that $SL_2(4)$ is embedded in $G$ as an $A_5$
moving $5$ points in the natural permutation action. So $SL_2(4)$ is contained
in $G$ up to conjugacy, and as before, it has an element of order $3$ acting
without fixed points on $V^\#$.  Hence, (d) holds in this case as well by
Higman's Theorem.
\end{proof}

For a vector space $E$ over the field with two elements, the next lemma
examines under rather strong hypotheses the structure of extensions of $E$ by
certain subgroups of the stabilizer in $GL(E)$ of a hyperplane.

\begin{lemma}\label{L:homocyclic}
Let $E \cong E_{2^{n+1}}$ ($n \geq 3$), $A = \Aut(E)$, $V \leq E$ with $V \cong
E_{2^{n}}$, $P = N_A(V)$, and $U = O_2(P)$. Let $L$ be a complement to $U$ in
$P$ acting decomposably on $E$, $x \in E - V$ the fixed point for the action of
$L$, and $G \leq L$. Let $H$ be an extension of $E$ by $UG$ with the given
action and let $X$ be the preimage in $H$ of $U$ under the quotient map $H \to
UG$. Assume that 
\begin{enumerate}
\item[(a)] $G$ acts transitively on $V^\#$; 
\item[(b)] $C_{GL(V)}(G) \leq G$; and
\item[(c)] $H^1(G,V) = 0$.
\end{enumerate}
Then there is a subgroup $Y$ of $X$ that is elementary abelian or homocyclic of
order $2^{2n}$ and a $G$-invariant complement to $\gen{x}$ in $X$.
\end{lemma}
\begin{proof}
Let $\bar{X} = X/V$. Since the commutator map
\begin{eqnarray}\label{E:commx}
\mbox{$[x, -]$ determines a $G$-equivariant linear isomorphism $X/E \to V$,}
\end{eqnarray}
$G$ is transitive on the nonzero vectors of $X/E$ by (a), and $\gen{\bar{x}}
\leq Z(\bar{X})$. Hence if $\bar{X}$ is not elementary abelian, then it is
extraspecial with center $\gen{\bar{x}}$, and $G$ preserves the squaring map
$\bar{X} \to Z(\bar{X})$. This is not the case, because $G$ is transitive on
the nonzero vectors of $X/E$ and $n \geq 3$. Therefore,
\begin{eqnarray}
\label{E:X/Velemab}
\text{$\bar{X}$ is elementary abelian}.
\end{eqnarray}

Assumption (c) now yields that there is a $G$-invariant complement $\bar{Y}$ to
$\gen{\bar{x}}$. Let $Y$ be the preimage of $\bar{Y}$ in $X$.  We claim that
$Y$ is abelian; assume on the contrary.  Then $[Y,Y]$ and $\mho^1(Y)$ are
contained in $V$ since $\bar{Y}$ is elementary abelian, and by assumption,
neither of these are trivial. Similarly, $V$ is contained in $Z(Y)$, which is
not $Y$. Therefore, 
\begin{eqnarray}
\label{E:Ynonabelian}
V = [Y,Y] = \Phi(Y) = \mho^1(Y) = Z(Y), 
\end{eqnarray}
by (a) and \eqref{E:commx}.

By \eqref{E:Ynonabelian}, the squaring map $\bar{Y} \to V$ is $G$-equivariant
linear isomorphism; let $\sqrt{-}$ be its inverse.  Then the map $\xi\colon V
\to V$ given by $\xi(v) = [x, \sqrt{v}]$ is a linear isomorphism commuting with
the action of $G$, and so $\xi \in \rho(G)$ by (b), where $\rho\colon G \to
GL(V)$ is the structure map. Let $g \in G$ map to $\xi^{-1}$ under $\rho$.
Then $y \mapsto [x, y^g]$ is the squaring map. This means that for each $y \in
Y$, we have $y^{gx} = y^2y^g$. Hence, for each pair $w$, $y \in Y$,
\[
w^2y^2w^gy^g = w^{gx}y^{gx} = (wy)^{gx} = (wy)^2w^gy^g
\]
which gives $w^2y^2 = (wy)^2 = w^2y^2[y,w]$. Thus $Y$ is abelian after all. It
follows that $\Omega_1(Y) = V$ or $Y$ by (a), and this completes the proof of
the lemma.
\end{proof}

We now examine the structure of the simple systems occupying the role of $\K$
in Theorem~\ref{T:main}. 
Let $T_0$ be a $2$-group isomorphic to a Sylow $2$-subgroup of $L_3(4)$.  This
is generated by involutions $t_1$, $t_2$, $a_1$, $a_2$, $b_1$, and $b_2$ such
that $Z(T_0) = \gen{t_1, t_2}$ and with additional defining relations:
\[
[a_1, a_2] = [b_1, b_2] = 1, \quad [a_1, b_1] = [a_2, b_2] = t_1, \quad [a_2,
b_1] = t_2, \quad [a_1, b_2] = t_1t_2.
\]

A Sylow subgroup of $M_{23}$ or $McL$ is isomorphic to a Sylow $2$-subgroup of
an extension of $L_3(4)$ by a field automorphism; this is a semidirect product
$T_0\gen{f}$ with
\[
f^2 = 1, \quad [a_1, f] = [a_2, f] = a_1a_2, \quad [b_1, f] = [b_2, f] =
b_1b_2, \quad [t_2, f] = t_1.
\]
A Sylow subgroup of 
$J_3$ is isomorphic to a Sylow $2$-subgroup of $L_3(4)$ extended by a unitary
(i.e., graph-field) automorphism; this is a semidirect product $T_0\gen{u}$
with
\[
u^2 = 1, \quad [a_1,u] = [u,b_1] = a_1b_1, \quad [a_2,u] = [u,b_2] = a_2b_2,
\quad [t_2, u] = t_1.
\]
A Sylow subgroup of $Ly$ is isomorphic to a Sylow $2$-subgroup of $\Aut(L_3(4))$;
this is a semidirect product $T_0\gen{f, u}$ with $[f,u] = 1$ and the relations
above.  

Denote by $T_1$ a $2$-group isomorphic to one of $T_0\gen{f}$, $T_0\gen{u}$, or
$T_0\gen{f,u}$.  Recall that the \emph{Thompson subgroup} $J(P)$ of a finite
$p$-group $P$ is the subgroup generated by the elementary abelian subgroups of
$P$ of largest order. 

\begin{lemma}\label{L:sporprelim}
Let $K$ be $M_{23}$, $McL$, 
$J_3$, or $Ly$, with Sylow $2$-subgroup $T_1$ as above. Then
\begin{enumerate}
\item[(a)] $Z(T_1) = \langle t_1 \rangle$ is of order $2$;
\item[(b)] $\A(T_1) = \{F_1, F_2\}$ where $F_1 = \gen{t_1,t_2,a_1,a_2}$ and
$F_2 = \gen{t_1,t_2,b_1,b_2}$, so that $J(T_1) = T_0$. Also, after suitable
choice of notation, one of the following holds:
\begin{enumerate}
\item[(i)] $K = M_{23}$, $McL$, or $Ly$, and $\Aut_K(F_1) \cong A_7$, or
\item[(ii)] $K = J_3$ and $\Aut_K(F_1) \cong GL_2(4)$. 
\end{enumerate}
\item[(c)]
There is $F \in \A(T_1)$ such that the pair $(\Aut_K(F), F)$ satisfies
assumptions (a)-(c) of Lemma~\ref{L:homocyclic} in the role of $(G,V)$. 
\item[(d)] All involutions in $\gen{t_1, t_2}$ are $\Aut_K(J(T_1))$-conjugate.
\end{enumerate}
\end{lemma}

\begin{proof}
Point (a) holds by inspection of the relations above.  Now $F_1$ and $F_2$ are
the elementary abelian subgroups of $T_0$ of maximal rank, and so to prove (b),
it suffices to show that each elementary abelian subgroup of maximal rank in
$T_1$ is contained in $T_0$. Set $L := L_3(4)$, and identify $\Inn(L)$ with
$L$. Write $\Inndiag(L) \geq L$ for the group of inner-diagonal automorphisms
of $L$. Then $\Inndiag(L)$ contains $L$ with index $3$, corresponding to the
size of the center of the universal version $SL_3(4)$ of $L$
\cite[Theorem~2.5.12(c)]{GLS3}. Also, $\Aut(L)$ is a split extension of of
$\Inndiag(L)$ by $\Phi_L\Gamma_L = \Phi_L \times \Gamma_L$, where $\Phi_L =
\gen{\phi} \cong C_2$ is generated by a field automorphism of $L$, and
$\Gamma_L = \gen{\gamma} \cong C_2$ is generated by a graph automorphism
\cite[Theorem~2.5.12]{GLS3}. By \cite[Theorems~4.9.1, 4.9.2]{GLS3}, each
involution of $\Aut(L)-L$ is $\Aut(L)$-conjugate to $\phi$, $\phi \gamma$, or
$\gamma$, and the centralizers in $L$ of these automorphisms are isomorphic to
$L_3(2)$, $U_3(2) \cong (C_3 \times C_3)Q_8$, and $Sp_2(4) \cong A_5$,
respectively, again by those theorems. These centralizers have $2$-ranks $2$,
$1$, and $2$, respectively. Since $T_0$ has $2$-rank $4$, this shows that
$J(T_1) \leq T_0$. From the relations used in defining $T_0$, each involution
in $T_0$ is contained in one of $F_1$ or $F_2$, so we conclude that $\A(T_1) =
\A(T_0) = \{F_1,F_2\}$.  The description of the automizers in (b) follows from
\cite[Table~1]{FinkelsteinM23} for $M_{23}$ and $McL$,
\cite[Lemma~3.7]{FinkelsteinHJ} for $J_3$, and \cite{Wilson1984} for Ly.
Now point (c) follows from (b) and Lemma~\ref{L:H1}, and point (d) follows from
(c) and Burnside's fusion theorem (i.e. the statement that the automizer of a
weakly $K$-closed subgroup of $T_1$ (which is $J(T_1)$ in this case) controls
the $K$-conjugacy in its center).  
\end{proof}

\begin{lemma}\label{L:sporautprelim}
Let $K$ be one of the sporadic groups $M_{23}$, $McL$,
$J_3$, or $Ly$, and let $T_1$ be a Sylow $2$-subgroup of $K$. Then 
\begin{enumerate}
\item[(a)] $\Out(K) = 1$ if $K \cong M_{23}$ or $Ly$, and $\Out(K) \cong C_2$
otherwise; and 
\item[(b)] for each involution $\alpha \in \Aut(K)-\Inn(K)$,
\begin{enumerate}
\item[(i)] $C_K(\alpha) \cong M_{11}$ if $K \cong McL$,
\item[(ii)] $C_K(\alpha) \cong L_2(17)$ if $K \cong J_3$; and
\end{enumerate}
\item[(c)] each automorphism of $K$ centralizing a member of $\A(T_1)$ is
inner. 
\end{enumerate}
\end{lemma}
\begin{proof}
Points (a) and (b) follow by inspection of Table~5.3 of \cite{GLS3}. By
Lemma~\ref{L:sporprelim}(b), the $2$-rank of $K$ is $4$, while each of the
centralizers in (b)(i-ii) is of $2$-rank $2$, so (c) holds. 
\end{proof}

\section{Preliminary lemmas}\label{S:prelim}

We now begin in this section the proof of Theorem~\ref{T:main}, and so we fix
the notation and hypotheses that will hold throughout the remainder of the
paper. 

Let $\F$ be a saturated fusion system over the $2$-group $S$, and let $x \in S$
be an involution. Assume that $\gen{x}$ is fully $\F$-centralized, that
$m(C_S(x)) = m(S)$, and that $F^*(C_\F(x)) = Q \times \K$, where $Q$ is cyclic.
Set $\C = C_\F(x)$ and $T = C_S(x)$, so that $\C$ is a saturated fusion system
over $T$ by the remark just after Lemma~\ref{L:fn}. Let $T_1$ be the Sylow
subgroup of $\K$, and set 
\[
R := Q \times T_1 \leq T. 
\]
Assume that $\K$ is the fusion system of one of the sporadic groups $K =
M_{23}$, 
$J_3$, $McL$, or $Ly$.  Since $K$ tamely realizes $\K$ in each case, the
quotient 
\begin{eqnarray}
\label{E:KT/Q}
\text{$T/R$ induces a $2$-group of outer automorphisms of $K$}
\end{eqnarray}
by Theorem~\ref{T:tame}. Arguing by contradiction, we assume 
\[
\text{$\K$ is not a component of $\F$}. 
\]
We fix the presentation in Section~\ref{S:structure} for $T_1$ in whichever
case is applicable, and we note that $\Omega_1(Q) = \gen{x}$ by assumption on
$Q$. 

\begin{lemma}\label{L:Jfn}
Notation may be chosen so that $T$ and $J(T)$ are fully $\F$-normalized.
\end{lemma}
\begin{proof}
We repeatedly use Lemma~\ref{L:fn}. Let $\alpha \in \mathfrak{A}(T)$. Then as
$|C_S(x^\alpha)| \geq |T^\alpha| = |T| = |C_S(x)|$, we have that $x^\alpha$ is
still fully $\F$-centralized. Thus we may assume that $T$ is fully
$\F$-normalized after replacing $x$, $T$, and $J(T)$ by their conjugates under
$\alpha$.  Now let $\beta \in \mathfrak{A}(J(T))$. Then as $N_S(T) =
N_{N_S(J(T))}(T)$, it follows that 
\[
|N_S(T^\beta)| = |N_{N_S(J(T)^\beta)}(T^\beta)| \geq |N_{N_S(J(T))^\beta}(T^\beta)| =|N_{N_S(J(T))}(T)^\beta| = |N_S(T)^\beta| = |N_S(T)|
\]
and so equality holds because $T$ is fully $\F$-normalized.  Hence $T^\beta$ is
still fully $\F$-normalized.  As before, $x^\beta$ is still fully
$\F$-centralized.
\end{proof}

\begin{lemma}\label{L:xnotwc} 
$\gen{x}$ is not weakly $\F$-closed in $T$.
\end{lemma}
\begin{proof}
Assume on the contrary, in which case $T = S$, otherwise $N_S(T)$ contains $T$
properly and moves $x$. It follows that $x \in Z(S)$, which is contained in the
center of every $\F$-centric subgroup. Hence $x \in Z(\F)$ by Alperin's fusion
theorem \cite[Theorem~A.10]{BrotoLeviOliver2003}, and we conclude that $\K$ is
a component of $\C = \F$, contrary to hypothesis.
\end{proof}

\begin{lemma}\label{L:centT1}
The following hold.
\begin{enumerate}
\item[(a)] $J(T) = J(R) = \gen{x} \times J(T_1)$; and
\item[(b)] $C_T(T_1) = Q\gen{t_1}$. 
\end{enumerate}
\end{lemma}
\begin{proof}
Suppose (a) does not hold. Choose $A \in \A(T)$ with $A \nleq \gen{x} \times
J(T_1)$. Then $A \nleq R$ by the structure of $R$, and $A$ acts nontrivially on
$K$ by \eqref{E:KT/Q}. In particular, $\Out(K) = 2$ and $m(C_R(A)) \leq 3$ by
Lemma~\ref{L:sporautprelim}. Hence, $m(A) \leq 4$ while $m(R) = 5$. This
contradicts the choice of $A$ and establishes (a).  By
Lemma~\ref{L:sporprelim}(a), $C_R(T_1) = Q\gen{t_1}$.  Also, $C_T(T_1) =
C_R(T_1)$ by part (a), so (b) is also established.  
\end{proof}

\begin{lemma}\label{L:rank-Z(S)} 
If $T=S$, then $\Omega_1Z(S) = \gen{x,t_1}$. If $T < S$, then $\Omega_1(Z(S)) =
\gen{t_1}$.  
\end{lemma} 
\begin{proof} 
Note first that $\Omega_1(Z(S)) \leq T$. By Lemma~\ref{L:sporautprelim}(c) and
\eqref{E:KT/Q}, $\Omega_1(Z(S)) \leq R$, and so $\Omega_1Z(S) \leq
\Omega_1(Z(R)) = \gen{x, t_1}$ by Lemma~\ref{L:sporprelim}(a).  Thus, the lemma
holds in case $T = S$.  In case $T < S$, let $a \in N_S(T)-T$ with $a^2 \in T$.
Note that $[J(T),J(T)] = [J(T_1),J(T_1)] = \gen{t_1,t_2}$ by
Lemma~\ref{L:centT1}(a). So $a$ normalizes $\Omega_1Z(T) = \gen{x, t_1}$ and
$\Omega_1Z(T) \cap [J(T),J(T)] = \gen{x, t_1} \cap \gen{t_1,t_2} = \gen{t_1}$,
but $a$ does not centralize $x$.  Thus, $\Omega_1Z(S) = \gen{t_1}$ as claimed.  
\end{proof}

\begin{lemma}\label{L:xnsimt1}
The following hold.
\begin{enumerate}
\item[(a)] $x$ is not $\F$-conjugate to $t_1$; and
\item[(b)] $x$ is conjugate to $xt_1$ if and only if $T < S$, and in this case
$x$ is $N_S(T)$-conjugate to $xt_1$. 
\end{enumerate}
\end{lemma}
\begin{proof}
For part (a), let $\varphi \in \F$ with $x^\varphi \in Z(T)$.  Assume first
that $T=S$.  By the extension axiom, $\varphi$ extends to $\tilde{\varphi} \in
\Aut_\F(T)$, which restricts to an automorphism of $J(T)$.
Lemma~\ref{L:centT1}(a) shows that $x \notin [J(T),J(T)]$, while $t_1 \in
[J(T),J(T)]$ from Lemma~\ref{L:sporprelim}(d). Hence, $x^\phi \neq t_1$; this
shows $x$ is not $\F$-conjugate to $t_1$ in this case.

Now assume that $T < S$. Then $x \notin Z(S)$, whereas $Z(S) = \gen{t_1}$ by
Lemma~\ref{L:rank-Z(S)}. Since $\gen{x}$ is fully $\F$-centralized by
assumption, we conclude that $x$ is not $\F$-conjugate to $t_1$ in this case
either. This completes the proof of (a).

If $T < S$, then $x$ is $N_S(T)$-conjugate to $xt_1$ by (a), while if $T = S$,
then (a) and Burnside's fusion theorem imply that $\gen{x}$ is weakly
$\F$-closed in $\Omega_1(Z(S)) = \gen{x,t_1}$. Thus, (b) holds. 
\end{proof}

\section{The $2$-central case}\label{S:2central}

In this section it is shown that $T < S$; that is, $x$ is not $2$-central.  We
continue the notation set at the beginning of Section~\ref{S:prelim}.

\begin{lemma}
\label{L:xwcR}
If $T = S$, then $\gen{x}$ is weakly $\F$-closed in $R$. 
\end{lemma}
\begin{proof}
Assume $T = S$. Then $\Omega_1(Z(S)) = \gen{x,t_1}$ by Lemma~\ref{L:rank-Z(S)}.
Using Burnside's fusion theorem and assumption, we see from
Lemma~\ref{L:xnsimt1} that 
\begin{eqnarray}\label{E:xwcZT1}
\text{$\gen{x}$ is weakly $\F$-closed in $\gen{x,t_1}$.}
\end{eqnarray}

By inspection of \cite[Table~5.3]{GLS3}, $K$ has one class of involutions.
Thus, there are exactly three $\C$-classes of involutions, namely $\{x\}$,
$(xt_1)^{\C}$, and $t_1^{\C}$. The lemma therefore holds 
by \eqref{E:xwcZT1}. 
\end{proof}

\begin{lemma}\label{L:TneqS} 
If $T=R$, then $T < S$. In particular, $T < S$ in case $\K$ is the fusion
system of $M_{23}$ or $Ly$. 
\end{lemma}
\begin{proof}
Assume $T = R$, and also to the contrary that $T = S$. By Lemma~\ref{L:xwcR},
$\gen{x} \leq Z(S)$ is weakly $\F$-closed, and so is fixed by each automorphism
of each $\F$-centric subgroup. Therefore, $\gen{x} \leq Z(\F)$, and so $\K$ is
a component of $\C = \F$, contrary to assumption. The last statement follows
then follows from Lemma~\ref{L:sporautprelim}(a). 
\end{proof}

\begin{lemma}\label{L:McL-cent}
Assume $\K$ is the fusion system of $McL$ or $J_3$.
Then $T < S$. 
\end{lemma}
\begin{proof}
Assume $T = S$. Then $R < T$ by Lemma~\ref{L:TneqS}.  Fix $f \in x^\F \cap
(T-R)$. The extension $\K\gen{f} := \K T_1\gen{f}$ of $\K$ is defined by
\cite[\S 8]{AschbacherGeneralized}, and $\K\gen{f}/Q$ is the $2$-fusion system
of $\Aut(McL)$ or $\Aut(J_3)$
by Theorem~\ref{T:tame}. Thus, $|T:R| = 2$ by Lemma~\ref{L:sporautprelim}(b).

Conjugating in $\K\gen{f}$ if necessary, we may assume $\gen{f}$ is fully
$\K\gen{f}$-centralized.  By Lemma~\ref{L:sporautprelim}(b), all involutions of
$T_1\gen{f}-T_1$ are $\K\gen{f}$-conjugate, and $C_{\K\gen{f}}(f) \cong \gen{f}
\times \F_2(M_{11})$ or $\gen{f} \times \F_2(L_2(17))$. 
In particular, $C_{T_1}(f)$ is semidihedral or dihedral, respectively, of order
$16$ and with center $\gen{t_1}$. Fix a four subgroup $V \leq C_{T_1}(f)$.
Then $f$ is conjugate to $ft_1$ (for example, by an element in the normalizer
of $C_{T_1\gen{f}}(f)$ in $T_1\gen{f}$), and hence is $\F$-conjugate to each
element of $fV$ by the structures of $\F_2(M_{11})$ and $\F_2(L_2(17))$.

Fix $\alpha \in \Hom_{\F}(\gen{f},\gen{x})$. By the extension axiom, $\alpha$
extends to a morphism, which we also call $\alpha$, defined on $\gen{f}V \leq
C_{T}(f)$.  Therefore, $x$ is $\F$-conjugate to each element in $xV^\alpha$.
Now the intersection $xV^\alpha \cap R = V^\alpha \cap R$ is nontrivial because
$|T:R| = 2$, so as $x$ is not itself in $V^\alpha$, we see that $x$ has a
distinct conjugate in $R$. This contradicts Lemma~\ref{L:xwcR} and completes
the proof.
\end{proof}

\section{Proof of Theorem~\ref{T:main}}

Continue the notation and hypotheses set at the beginning of
Section~\ref{S:prelim}.  In addition, we fix $F \in \A(T_1)$ satisfying
assumptions (a)-(c) of Lemma~\ref{L:homocyclic}, as guaranteed by
Lemma~\ref{L:sporprelim}(c), and set $E := \gen{x}F$. 
Then $E \in \A(T)$ by Lemma~\ref{L:centT1}(a), and so
\begin{eqnarray}
\label{E:mt5}
m(T) = 5.
\end{eqnarray}

In this section, we finish the proof of Theorem~\ref{T:main}, by showing that
the hypotheses of Lemma~\ref{L:homocyclic} hold for a model of the normalizer
in $\F$ of an appropriate $\F$-conjugate of $E$. Via Lemma~\ref{L:H1}(d), this
forces the 2-rank of $S$ to be at least $8$, contrary to the hypothesis that
$m(S) = m(T)$. 

By Lemmas~\ref{L:TneqS} and \ref{L:McL-cent},
\[
T < S.
\]

\begin{lemma}\label{L:xsimxt1}
$|\Aut_\F(T):\Aut_\C(T)| = 2$.
\end{lemma}
\begin{proof}
Represent $\Aut_\F(T)$ on $\Omega_1Z(T) = \gen{x, t_1}$ and apply
Lemma~\ref{L:xnsimt1}.
\end{proof}

\begin{lemma}
\label{L:notBaum}
The following hold.
\begin{enumerate}
\item[(a)] $|\Aut_\F(J(T)):\Aut_\C(J(T))| = 4$ and 
\item[(b)] $x^{\Aut_\F(E)} = xF$, and so $|\Aut_\F(E):\Aut_\C(E)| = 16$.
\end{enumerate}
\end{lemma}
\begin{proof}
Represent $\Aut_\F(J(T))$ on $\Omega_1Z(J(T)) = \gen{x,t_1,t_2}$.  Now
$\Aut_\C(J(T)) = C_{\Aut_{\F}(J(T)}(x)$, and the former is transitive on
$\Omega_1ZJ(T_1)^\#$ by Lemma~\ref{L:sporprelim}(d). Also, since $x$ is
$N_S(T)$-conjugate to $xt_1$, we conclude from Lemma~\ref{L:xnsimt1} that
$x^{\Aut_\F(J(T))} = xZJ(T_1)$ is of size $4$. Thus, (a) holds. 

Similarly to (a), we have $\Aut_\C(E) = C_{\Aut_\F(E)}(x)$, and the former is
transitive on $F^\#$ by choice of $F$. From Lemma~\ref{L:centT1}(a) and
Lemma~\ref{L:sporprelim}(a), $|\A(T)| = 2$. By part (a) and Lemma~\ref{L:Jfn},
$|N_S(J(T)):T| = 4$. Representing $N_S(J(T))$ on $\A(T)$, we see that the
kernel has index at most $2$, so there is an element of $N_S(J(T))-T$ that
normalizes $E$.  In particular, $N_T(E) < N_S(E)$, and so $x$ is
$\Aut_\F(E)$-conjugate to a member of $xF^\#$. Now by choice of $F$, another
appeal to Lemma~\ref{L:xnsimt1} yields that $x^{\Aut_\F(E)} = Fx$ has size 16,
which establishes (b).
\end{proof}

\begin{lemma}\label{L:Q=x}
The following hold:
\begin{enumerate}
\item[(a)] $Q = \gen{x}$.
\item[(b)] $E$ is $\F$-centric.
\end{enumerate}
\end{lemma}
\begin{proof}
Suppose on the contrary that $Q > \gen{x}$ and choose $w \in Q$ with $w^2 = x$.
Fix $a \in N_S(T)-T$ such that $a^2  \in T$. Then $x^a = xt_1$ and also
$(w^a)^2 = xt_1$. Further $\gen{w^a}$ is normal in $T$ since $\gen{w}$ is.
Thus, $[\gen{w^a}, T_1] \leq \gen{w^a} \cap T_1 = 1$, whereas $C_{T}(T_1) =
Q\gen{t_1}$ by Lemma~\ref{L:centT1}(b). It follows that $\gen{xt_1} =
\mho^1(\gen{w^a}) \leq \Omega_1\mho^1(Q\gen{t_1}) = \gen{x}$, a contradiction that
establishes (a). 

Let $E_0$ be one of the two elementary abelian subgroups of rank $5$ in $T$,
and set $F_0 = E_0 \cap J(T_1)$.  Then $E_0 = \gen{x}F_0$ contains $x$, and so
$C_S(E_0) = C_T(E_0)$. By Lemma~\ref{L:sporautprelim}(c), $C_T(E_0) = C_R(E_0)
= QF_0$.  Hence $C_S(E_0) = QF_0 = E_0$ by part (a).

We can now prove (b). Fix $\alpha \in \mathfrak{A}(E)$.  Since $\gen{x}$ is
fully centralized, the restriction of $\alpha^{-1}$ to $\gen{x^\alpha}$ has an
extension $\beta\colon C_S(x^\alpha) \to C_S(x) = T$, which is defined on
$C_S(E^\alpha)$.  Thus, setting $E_0 := C_S(E^\alpha)^\beta \leq T$, we see
from the previous paragraph that $|C_S(E)| = |C_S(E_0)| = |E_0| =
|C_S(E^\alpha)|$, so that $C_S(E^\alpha) = E^\alpha$. As $E^\alpha$ is fully
$\F$-normalized and contains its centralizer in $S$, this means that $E$ is
$\F$-centric, as claimed.
\end{proof}

Since we will be working in $N_\F(E)$ for the remainder, we may assume, after
replacing $E$ by an $\F$-conjugate if necessary, that $E$ is fully
$\F$-normalized. Hence $E \in \F^{fc}$ by Lemma~\ref{L:Q=x}(b). Fix a model $H$
for $N_\F(E)$ (cf. \S\S\ref{SS:model}).

\begin{lemma}\label{L:punchline}
$H$ satisfies the hypotheses of Lemma~\ref{L:homocyclic}.
\end{lemma}
\begin{proof}
Set $\tilde{G} = \Aut_\C(E)$ and observe that $\Aut_\F(E) \cong H/E$ by
Lemma~\ref{L:Q=x}(b). Thus $\tilde{G}$ contains $G := A_7$ or $GL_2(4)$ with
index $1$ or $2$. As $G$ acts transitively on $F^\#$ and centralizes $x$, it
follows from Lemma~\ref{L:notBaum}(b) that $xF$ and $F^\#$ are the orbits of
$\Aut_\F(E)$ on $E^\#$.  Hence $\Aut_\F(E) \leq N_{\Aut(E)}(F)$, a nontrivial
split extension of an elementary abelian $2$-group $U$ of order $16$ by $GL(F)$
with the standard action.

We claim that $U \leq \Aut_\F(E)$. Suppose that this is not the case. Now $G$
acts transitively on $F^\#$ and the commutator map $[x,-]$ defines an
isomorphism of $G$-modules from $U$ to $F$, so $G$ acts transitively on $U^\#$.
Since $U$ is normalized by $\Aut_\F(E)$, we see that $\Aut_\F(E) \cap U = 1$.
In particular, $\Aut_\F(E)$ embeds into $GL(F)$. Now $G \cong A_7$ or $GL_2(4)$
in the cases under consideration, and by Lemma~\ref{L:notBaum}(b), $\Aut_\F(E)$
is therefore a subgroup of $GL(F)$ containing $G$ with index $16$ or $32$.
However, $A_7$ has index $8$ in $GL_4(2)$, and $GL_2(4)$ is contained with
index 2 in a unique maximal subgroup of $GL_4(2)$, a contradiction. Therefore,
$U \leq \Aut_\F(E)$ as claimed. 

It has thus been shown that $\Aut_\F(E)$ contains a subgroup with index $1$ or
$2$ that is a split extension of $U = O_2(N_{\Aut(E)}(F))$ by $G$. Thus, $H$
has a subgroup of index $1$ or $2$ that is an extension of $E$ by $UG$.
Assumptions (a)-(c) of Lemma~\ref{L:homocyclic} hold via
Lemma~\ref{L:sporprelim}(c) by the choice of $F$.
\end{proof}

\begin{proof}[Proof of Theorem~\ref{T:main}]
Keep the notation of the proof of Lemma~\ref{L:punchline}. By that lemma and
Lemma~\ref{L:homocyclic}, there is a $G$-complement $Y$ to $\gen{x}$ in
$O_2(H)$ that is homocyclic of order $2^8$ with $\Omega_1(Y) = F$, or
elementary abelian of order $2^8$.  Now $G$ is isomorphic to $A_7$ or $GL_2(4)$
with faithful action on $F$, so the former case is impossible by
Lemma~\ref{L:H1}. Hence, $m_2(T) = 5 < 8 \leq m_2(S)$, contrary to hypothesis.
\end{proof}

\bibliographystyle{amsalpha}{}
\bibliography{/home/cpsmth/s05jl6/work/math/research/mybib}
\end{document}